\documentclass{amsart}
\newtheorem{Theo}{Theorem}
\newtheorem{Prop}{Proposition}
\newtheorem{Cor}{Corollary}
\newtheorem{Lem}{Lemma}
\newcommand{\Z}{\mathbb{Z}}
\newcommand{\N}{\mathbb{N}}
\newcommand{\R}{\mathbb{R}}

\begin{document}
\title{All large primes have property D}
\author[J.-C. Schlage-Puchta]{Jan-Christoph Schlage-Puchta}
\begin{abstract}
Let $p$ be a prime number, $a_1, a_2, \ldots a_{4p-4}$ a sequence of elements in $(\Z/p\Z)^2$, which does not contain a subsequence of length $p$ which adds up to 0. We show that if $p$ is sufficiently large, then the sequence contains exactly four different elements.
\end{abstract}
\maketitle

\section{Introduction and Results}

Let $(G, +)$ be a finite abelian group with neutral element $0$.. A sequence $a_1, \ldots, a_n$ of elements in $G$ is called a zero-sum sequence, if $a_1+\dots +a_n=0$. A sequence $a_1,\dots, a_n$ is called zero-sum free, if there is no subsequence $a_{i_1}, a_{i_2}, \ldots, a_{i_k}$, $1\leq i_1<i_2<\dots<i_k\leq n$, which is a zero-sum sequence. The basic questions surrounding zero-sums are on one hand to determine the least integer $n$, such that every sequence of length $n$ contains a zero-sum, or the least integer $n$ such that every sequence contains a zero-sum subsequence with some restriction on the length of the sub-sequence. On the other hand we have the inverse questions, that is, given a zero-sum free sequence  of maximal length, can we determine the structure of the sequence?

For an integer $n$ and $G=(\Z/n|Z)^2$, the systematic investigation of inverse questions began with the work of Gao and Geroldinger \cite{GG}. They defined that an integer $n$ has
\begin{itemize}
\item {\em Property B}, if every zero-sum free sequence of length $2n-2$ in $(\Z/n|Z)^2$ contains an element with multiplicity at least $p-2$;
\item {\em Property C}, if every sequence of length $3n-3$ in $(\Z/n\Z)^2$, that does not contain a zero-sum subsequence of length $\leq p$, is of the form $a^{p-1}b^{p-1}c^{p-1}$;
\item {\em Property D}, if every sequence of length $4n-4$ in $(\Z/n\Z)^2$, that does not contain a zero-sum subsequence of length $p$, is of the form $a^{p-1}b^{p-1}c^{p-1}d^{p-1}$;
\end{itemize}
and showed that Property C and D are multiplicative, that is, if $n$ and $m$ have this property, then so does $nm$. In particular, it suffices to check these properties for prime numbers.

There has been a lot of research around these properties, which we will not summarize here. We only mention the following milestones: Reiher \cite{Reiher1} proved Kemnitz' conjecture, that is, every every sequence of length $4p-3$ in $(\Z/p\Z)^2$ contains a zero-sum of length $p$,. In other words, property D at least predicts the correct maximal length of a sequence without a zero-sum of length $p$. Gao, Geroldinger and Grynkiewicz \cite{Inverse III} showed that Property B is multiplicative. Gao and Geroldinger \cite{GG0} showed that property B implies property C. Reiher \cite{Reiher2} showed that all prime numbers satisfy property B, taken together these results imply that all integers have property B and C.

The goal of this note is to prove the following.
\begin{Theo}
\label{thm:main}
There exists a constant $p_0$, such that every prime number $p>p_0$ satisfies property D.
\end{Theo}

The proof falls into two main parts and a short conclusion. In the first part we show that if $A$ is a sequence of length $4p-4$ without a zero-sum subsequence  of  length $p$, then, after a suitable change of variables  all but a tiny fraction of the elements of the sequence have coordinates of absolute value bounded by some constant $C$, independently of $p$. This part of the proof uses methods from Fourier theory, in particular techniques used by Alon and Dubiner \cite{AD} and exponential sums. The second part uses idea from convex geometry to show that a  real approximation of the problem is solvable. Finally, we show that the approximation is good enough to yield a proof of the full problem.

A central idea that turns the problem into a combinatorial search over a finite set of instances is taken from \cite{Numerik}, where Property B was established for all $p\leq 23$. One method to reduce the computational effort was the following: The set of sums of subsequences of $x-y\;x\;x+y$ is at least as large as the set of sums of subsequences of $x^3$, hence, if we want to check whether a certain sequence contains a zero-sum subsequence, we can replace all occurrences of triples $x-y\;x\;x+y$ by $x^3$. When looking for sequences with a given structure such that there is no zero-sum subsequence, or the set of subsequence sums is in other ways restricted, this argument greatly helps to prune the search tree.

We would like to do the same for elements $x+y\;x-y$. However, the sequence $x^2$ now has the subsequence sums $\{1, x, 2x\}$, and in general $x$ is not a subsequence sum of $x+y\;x-y$. If we replace the sequence $(x+y)^k\;(x-y)^k$ by $x^{2k}$, the same problem arises. However, if we can represent $x$ by some other means as a subsequence sum, we can almost remedy this, so the subsequence sums of $(x+y)^k\;(x-y)^k$ are almost a superset of the subsequence sums of $x^{2k}$. To give an exact meaning to ``almost", in Section~\ref{sec:geo} we introduce real sequences, that is, sequences where multiplicities are non-negative real numbers. We obtain a geometric problem that we dubbed the real version of Property~D, see Proposition~\ref{Prop:Dreal}. 

\section{Fourier theoretic arugments}
\label{sec:Fourier}

Our argument uses a sequence of approximations. In each step we reduce the sequence a bit. As there are sequences of length $4p-6$ in $(\Z/p\Z)^2$ without zerosums of length $p$, where all multiplicities are $\leq p-2$, this approach could at best show that some multiplicity is pretty large. Fortunately, this is sufficient in view of the following.
\begin{Lem}
Let $p$ be a prime number which has property C, and let $A$ be a sequence of length $4p-4$ that does not contain any zero sum subsequence of length $p$. Then either $A$ contains an element with multiplicity $p-1$ or all elements in $A$ have multiplicity at most $\frac{p}{2}$.
\end{Lem}
\begin{proof}
Property $D$ trivially holds for $p=2$, so we may assume that $p$ is odd.
Let $A$ be a sequence of length $4p-4$, which does not contain a zero-sum of length $p$, contains no element with multiplicity $p-1$, and contains an element with multiplicity $>\frac{p}{2}$. 
As $p$ has property $C$, we conclude that every subsequence of length $\geq 3p-3$ contains a zero-sum of length $\leq p$. As $A$ contains no zero-sum subsequence of length $p$, we even have that there exists a zero-sum of length $\leq p-1$. If $Z_1$ has length

We may assume without loss of generality that $(0,0)$ occurs with multiplicity $m>\frac{p}{2}$. Let $A_1$ be the sequence obtained from $A$ by deleting all occurrences of $(0,0)$. Then $A_1$ has length larger than $3p-3$, hence, $A$ contains zero-sum subsequences of length $\leq p-1$. Pick one such sequence $Z_1$ of maximal length. If $m+|Z_1|\geq p$, then we can form $Z_1(0,0)^{p-|Z_1|}$ and obtain a zero-sum of length $p$. If $m+|Z_1|<p$, we form $A_2=A_1\setminus Z_1$. Then 
\[
|A_2|=|A|-(m+|Z_1|)\geq 4p-4 - (p-1) = 3p-3,
\] 
hence, $A_2$ contains a zero-sum subsequence of length $\leq p-1$, pick one such sequence $Z_2$ of maximal length. Then $Z_1Z_2$ is a zero-sum subsequence, as $Z_1$ had maximal length among all sequences of length $\leq p-1$, we obtain $|Z_1Z_2|>p$. But then $|Z_1|>\frac{p}{2}$, and therefore
\[
p=\frac{p}{2}+\frac{p}{2} < m+|Z_1|<p,
\]
which is absurd.
\end{proof}
For the remainder of this section we therefore fix a sequence $A$ of length $4p-4$ without a zero-sum of length $p$ and all multiplicities $<\frac{p}{2}$.

Our first result is essentially contained in the work of Alon-Dubiner \cite{AD}.

\begin{Lem}
\label{Lem:AD}
For every $\epsilon>0$ and $d\in\N$ there exists some $W$ such that for all sequences $A$ over $(\Z/p\Z)^d$ of length $\geq \epsilon p$, such that no hyperplane contains more than $\frac{\epsilon p}{W}$ elements of the sequence, we have that all elements of $(\Z/p\Z)^d$ can be represented as the sum of a subsequence of $A$ of length $\lceil\epsilon p/2\rceil$.
\end{Lem}

\begin{Lem}
\label{Lem:line 3/2}
Let $B$ be a sequence over $\Z/p\Z$ of length $\geq \frac{3}{2}p$ and maximal mupltipicity $\leq \frac{p}{2}$. Then every element of $B$ is the sum of a subsequence of length $p$.
\end{Lem}
\begin{proof}
It follows from Lemma~\ref{Lem:AD} that there exists some $\delta>0$ such that our claim holds provided that the subsequence of $B$ consisting of all elements of $B$ with multiplicity $<\delta p$ has length at least $0.01p$. Now suppose that 0, 1 and $t$ occur with multiplicity $\geq\delta p$, and $t$ cannot be represented by a fraction with numerator and denominator of absolute value $\leq \frac{4}{\delta}$. Then the residues $0\bmod p, t\bmod p, 2t\bmod p, \ldots, \lfloor\frac{\delta p}{2}\rfloor t\bmod p$ partition $\{0, \ldots p\}$ into intervals, none of which is longer than $\frac{\delta}{2}p$. Hence, every element of $\Z/p\Z$ can be represented as the sum of a subseqeunce of $B$ of length $2\lfloor\frac{\delta p}{2}\rfloor$, and our claim follows. If this is not the case, then there is some $C$ such that after a suitable linear transformation all but $0.01p$ of all elements of $B$ occur with multiplicity $\geq\delta p$ and are in the interval $[-C, C]$. We may further assume that these elements, when viewed as elements of $\Z$, have no non-trivial common divisor. We conclude that when viewed as a subset of $\Z$, the set of all elements that can be represented differs from the interval $I=[a,b]$ starting at the smallest integer $a$ representable as such a subsequence sum going to the largest integer $b$ representable by a finite set of elements contained in $[a, a+M]\cup[b-M, b]$, where $M$ is some constant depending on $C$, and ultimately depending on nothing at all.

The fact that there are at least $1.49$ elements in $[-C, C]$, and no element occurs with multiplicity $>\frac{p}{2}$ no implies that these elements represent an interval of length at least $0.98p$. Furthermore, $b-a>1.01p$, unless there are two consecutive elements with multiplicity $\geq 0.47p$. But then the set of all integers representable as subsequence sums of length $p$ are actually the full interval $I$, and our claim follows.
\end{proof}
\begin{Cor}
\label{Cor:lines thin}
No line contains $\frac{3p}{2}$ elements of $A$.
\end{Cor}
\begin{Lem}
For every $\epsilon>0$ there exists constant $\delta>0$ such that for every sequence $A$ of length $\geq \frac{15}{4}p$ which has no element with multiplicity $>\frac{p}{2}$ there either exists a zero-sum subsequence of length $p$, or there exists a subsequence $A'$ of length $\geq |A|-\epsilon p$, in which all elements have multiplicity $\geq\delta p$.
\end{Lem}
\begin{proof}
If $A$ is a counterexample to the lemma, then we can pick a subsequence $B$ of $A$ of length $\epsilon p$, such that no element in $B$ has multiplicity $\geq\delta p$, where $\delta$ may be chosen depending on $\epsilon$ arbitrarily small. 
By Lemma~\ref{Lem:AD} there exists a constant $W=W(\epsilon)$ such that we can pick a subsequence $B'$ of $B$ of length $\geq \frac{\epsilon}{2}p$ which is covered by $W$ lines. Hence, we can choose a line $\ell$ that contains a sequence $C$ of length $\geq \frac{\epsilon}{2W}p$ elements of $A$, but no element with multiplicity $\geq\delta p$. We can now apply Lemma~\ref{Lem:AD} again to find that every element on a line $\ell'$ parallel to $\ell$ can be represented as the sum of a subsequence of $C$ of length $\left\lceil\frac{\epsilon}{2W}p\right\rceil$, provided that $\delta$ is sufficiently small.

Without loss of generality we may assume that $\ell'$ is parallel to the first coordinate axis. Let $\varphi:(\Z/p\Z)^2\rightarrow\Z/p\Z$ be the projection to the second coordinate. The projection $P=\pi(B\setminus B')$ has length $\geq 3.74p$, and no element has multiplicity $\geq\frac{3}{2}p$. If we can pick a subsequence $Q$ of $P$ of length $\geq\frac{3}{2}$ and maximal multiplicity $\leq\frac{p}{2}$, then we can apply Lemma~\ref{Lem:line 3/2} and are done. If this is not possible, let $m_1, m_2, \ldots, m_k$ be the list of all multiplicities $>\frac{p}{2}$ in $P$. If we take these $k$ elements with multiplicity $\frac{p-1}{2}$ each, then we obtain a sequence $Q$ as desired, unless $\frac{p-1}{2}k + 3.74-(m_1+\dots+m_k)\leq\frac{3p}{2}$. Clearly the left hand side of this inequality is decreasing in each of the $m_i$, hence, the left hand side is at least the value obtained if we set $m_i=\frac{3p-1}{2}$. In particular our claim holds for $k\leq 2$. On the other hand, for $k\geq 3$ it is trivial, and the proof is complete.
\end{proof}

\begin{Lem}\label{Lem:strip}
For every $\epsilon>0$ there exists a constant $C$ such that if $A$ is a sequence of length $4p-4$ with no zero-sum subsequence of length $p$ and maximal multiplicity $\leq\frac{p}{2}$, then after a suitable transformation of coordiatnes all but $\epsilon p$ of the elements of $A$ have first coordinate in the range $[-C, C]$.
\end{Lem}
\begin{proof}
Let $x_1, x_2, \ldots, x_n$ be the points in $(\Z/p\Z)^2$ that occur with multiplicity $\geq\delta p$. Put
\[
N=\left\lfloor\min\left(\frac{p}{\binom{n}{2}}, \frac{\delta p}{n-1}\right)\right\rfloor.
\]
For every $\alpha\in(\Z/pZ)^2$ define the exponential sum
\[
S(\alpha) = \prod_{1\leq i<j\leq n}\sum_{\nu=0}^N \eta(\nu\alpha\cdot (x_i-x_j)),
\]
where $\eta$ is a fixed non-trivial additive character of $\Z/p\Z$. The number of representations of an element $x\in(\Z/p\Z)^2$ in the form $x = \sum_{1\leq i<j\leq n} \nu_{ij}(x_i-x_j)$ with $0\leq \nu_{ij}\leq N$ is
\begin{multline}
\label{eq:Fourier}
\frac{1}{p^2}\sum_{\alpha\in(\Z/p \Z)^2} S(\alpha) \eta(-\alpha\cdot x)= \frac{(N+1)^{\binom{n}{2}}}{p^2} + \frac{1}{p^2}\underset{\alpha\neq(0,0)}{\sum_{\alpha\in(\Z/p \Z)^2}} S(\alpha)\eta(-\alpha\cdot x)\geq\\
 \frac{(N+1)^{\binom{n}{2}}}{p^2} - \frac{1}{p^2}\underset{\alpha\neq(0,0)}{\sum_{\alpha\in(\Z/p \Z)^2}} \prod_{1\leq i<j\leq n} \left|\sum_{\nu=0}^N \eta(\nu\alpha\cdot (x_i-x_j))\right|.
\end{multline}
We may choose a basis of $(\Z/p\Z)^2$ such that $\eta(u)=e^{\frac{2\pi i u}{p}}$. Then we get
\[
\left|\sum_{\nu=0}^N \eta(\nu\alpha\cdot (x_i-x_j))\right| \leq \min\left(N+1, \frac{2}{|1-e^{2\pi i\alpha\cdot (x_i-x_j)/p}|}\right) =: Z(\alpha\cdot(x_i-x_j)).
\]
Now assume the lemma is false. Then, for every $C$ we can find a sequence $A$ such that $A$ contains no zero-sum of length $p$, and no linear transform maps the support of $A$ to the strip $[-C, C]\times\Z/p/Z$. In particular after a suitable transformation we may assume that $A$ contains $(0,0)^{\lfloor\delta p\rfloor}(0,1)^{\lfloor\delta p\rfloor}(1,0)^{\lfloor\delta p\rfloor}(s,t)^{\lfloor\delta p\rfloor}$, where $s, t$ do not satisfy a linear relation of the form $as+bt+c\equiv 0\pmod{p}$ with $|a|, |b|, |c|\leq C$.

If we estimate trivially all factors in (\ref{eq:Fourier}) with the exceptions of those where $x_i=(0,0)$, and $x_j$ is one of $(1,0)$, $(0,1)$, $(s,t)$, we see that the number of representations of an element $x$ is at least
\[
\frac{(N+1)^{\binom{n}{2}}}{p^2} - \frac{(N+1)^{\binom{n}{2}-3}}{p^2}\underset{\alpha\neq(0,0)}{\sum_{\alpha\in(\Z/p \Z)^2}}  Z(\alpha\cdot(1,0)) Z(\alpha\cdot(0,1)) Z(\alpha\cdot(s,t)),
\]
so it is sufficient to show that for $C$ sufficienlty large we have
\[
\underset{\alpha\neq(0,0)}{\sum_{\alpha\in(\Z/p \Z)^2}}  Z(\alpha\cdot(1,0)) Z(\alpha\cdot(0,1)) Z(\alpha\cdot(s,t)) \leq \frac{1}{2}(N+1)^3.
\]
With this goal in mind we consider $\delta$ as small but fixed, whereas $C$ tends to infinity. We will show that the sum in question actually is $o(p^3)$, so if we pick $C$ large enough, our claim follows from $p\ll N$, as $\delta$ is considered fixed.

The contribution of summands where $\alpha\cdot(1,0)=0$ is
\[
\ll (N+1) \sum_{\nu=1}^{p-1} \frac{1}{\|\nu/p\|} \frac{1}{\|\nu t/p\|},
\]
where $\|\cdot\|$ denotes the difference to the nearest integer. 
By assumption we have that one of $|\nu/p\|, |\nu t/p\|$ is larger than $C$, so
\[
\sum_{\nu=1}^{C} \frac{1}{\|\nu/p\|} \frac{1}{\|\nu t/p\|} \leq \sum_{\nu=1}^{C} \frac{p}{\nu} \frac{p}{C} \ll \frac{\log C}{C}\ p^2,
\]
which, for sufficiently large $C$, is negligible. For the remaining sum note that as $\nu$ runs over $\{1, \ldots, p-1\}$, $\frac{1}{\|\nu t/p\|}$ attains each value in $\{\frac{p}{1}, \ \frac{p}{2}, \ldots, \frac{p}{(p-1)/2}\}$ exactly twice, so by the rearrangement inequality we get
\[
\sum_{\nu=C+1}^{p-C-1} \frac{1}{\|\nu/p\|} \frac{1}{\|\nu t/p\|} \leq 2\sum_{\nu=C}^{\frac{p-1}{2}} \frac{p}{\nu} \frac{p}{\nu-C} \ll \frac{\log C}{C}\ p^2,
\]
which is also sufficient. We conclude that we may neglect all terms on the line orthogonal to $(1,0)$, and the same argument applies to the lines orthogonal to $(0,1)$ and $(s,t)$, and by symmetry it suffices to bound the sum
\[
\underset{\nu s+\mu t\not\equiv 0\pmod{p}}{\sum_{\nu,\mu=1}^{\frac{p-1}{2}}}\frac{1}{\nu\mu|\nu s+\mu t\bmod p|},
\]
where $\bmod$ denotes the operator mapping an integer to the residue of least absolute value.
We again use the rearrangement inequality to obtain
\[
\underset{\nu s+\mu t\not\equiv 0\pmod{p}}{\sum_{C\leq\nu\leq\mu\leq\frac{p-1}{2}}}\frac{1}{\nu\mu|\nu s+\mu t\bmod p|} \leq \underset{\nu s+\mu t\not\equiv 0\pmod{p}}{\sum_{C\leq\nu\leq\mu\leq\frac{p-1}{2}}}\frac{1}{\nu\mu(\mu-\nu+1)}\ll\sum_{\nu=C}^{\frac{p-1}{2}}\frac{\log\nu}{\nu^2} \ll \frac{\log C}{C}
\]
and
\[
\underset{\nu s+\mu t\not\equiv 0\pmod{p}}{\sum_{\nu\leq C\leq\mu\leq\frac{p-1}{2}}}\frac{1}{\nu\mu|\nu s+\mu t\bmod p|} \leq \log C \sum_{\mu=C}^{\frac{p-1}{2}} \frac{1}{\mu(\mu-C+1)}\ll\frac{\log^2 C}{C}.
\]
Finally, if $1\leq\nu, \mu\leq C$, then $\nu s+\mu t\bmod p|>C$, and we obtain
\[
\underset{\nu s+\mu t\not\equiv 0\pmod{p}}{\sum_{\nu,\mu\leq C}}\frac{1}{\nu\mu|\nu s+\mu t\bmod p|} \leq \frac{1}{C}\sum_{\nu,\mu\leq C}\frac{1}{\nu\mu}  \ll\frac{\log^2 C}{C}.
\]
We conclude that the number of representations of a point $x$ is $\left(1+\mathcal{O}\left(\frac{\log^2 C}{C}\right)\right)\frac{(N+1)^{\binom{n}{2}}}{p^2}$, which for $C_1$ sufficiently large certainly becoms positive.
\end{proof}
\begin{Lem}
For every $\epsilon>0$ there exists a constant $C_2$ such that if $A$ is a sequence of length $4p-4$ with no zero-sum subsequence of length $p$ and maximal multiplicity $\leq\frac{p}{2}$, then after a suitable transformation of coordinates all but $\epsilon p$ of the elements of $A$ have both coordinates in the range $[-C_2, C_2]$.
\end{Lem}
\begin{proof}
In view of Lemma~\ref{Lem:strip} we may assume that all elements of the sequence are in a strip $[-C_1, C_1]\times \Z/p\Z$, and all elements occur with multiplicity $>\delta p$. We can translate $A$ and multiply the second coordinate by some element, such that $A$ contains $(0,0)$ and $(r, 1)$ for some $r\in[-2C_1, 2C_1]$, and $A$ is contained in $[-2C_1, 2C_1]\times \Z/p\Z$. If we cannot multiply the second coordinate by some number, such that $A$ is supported on $[-2C_1, 2C_1]\times[M!C_1, M!C_1]$, then there exists some element $(u,v)$ in $A$, such that $v$ cannot be represented as a fraction with numerator and denominator of absolute value $\leq M$. Pick such an element $(u,v)$.  Finally pick an element $(s,t)$ which is not on one of the lines through $(0,0)$ and one of $(r,1)$, $(u,v)$. We claim that no matter how small $\delta$ is, we can always pick $M$ so large that the set of sums of subsequences of $(0,0)^{\lfloor\delta p\rfloor}(1,r)^{\lfloor\delta p\rfloor}(u,v)^{\lfloor\delta p\rfloor}(s,t)^{\lfloor\delta p\rfloor}$ of length $\lfloor\delta p\rfloor$ contains a line parallel to $\{(0,y):y\in\Z/p\Z\}$.

We may assume that $v/t$ cannot be represented as a fraction with numerator and denominator $\leq M!C_1$, for otherwise we can simply swap the r\^ole of $(1,r)$ and $(u,v)$.

Note that $|r|, |s|, |u|\leq 2C_1$. Hence, the set of sums of subsequences of $(0,0)^{\lfloor\delta p\rfloor}$ $(1,r)^{\lfloor\delta p\rfloor}$ $(u,v)^{\lfloor\delta p\rfloor}$ $(s,t)^{\lfloor\delta p\rfloor}$ of length $\lfloor\delta p\rfloor$ is at least as large as the set of sums of subsequences of $(us,usr)^{\lfloor\delta p\rfloor}$ $(us,sv)^{\lfloor\delta p\rfloor}$ $(us,ut)^{\lfloor\delta p\rfloor}$ of length at most $\frac{\delta p}{4C_1^2}$. It therefore suffices to show that every element of $\Z_p$ can be written as the sum of a subsequence of $(sv-usr)^{\lfloor\delta p\rfloor}(ut-usr)^{\lfloor\delta p\rfloor}$ of length at most $\frac{\delta p}{4C_1^2}$. However, this follows immediately from the fact that $\frac{sv-usr}{ut-usr}$ is not a fraction with both numerator and denominator small.
\end{proof}

\section{The real version of Property D}
\label{sec:geo}

In this section we solve a continuous analogue of property D. Let $A=a_1^{m_1}a_2^{m_2}\cdots a_k^{m_k}$ be a sequence over $\Z^2$. 
We want to show that under suitable conditions there exist integers $e_1, \ldots, d_k$, such that $0\leq e_i\leq m_i$, $\sum_{i=1}^k e_i=p$ and $\sum_{i=1}^k e_i a_i$ has both coordinates divisible by $p$.
To do so we consider the set $S$ of all elements that can be represented as a subset sum of length $p$. If the multiplicities $m_i$ are pretty large, the coordinates of the $a_i$ are quite small, and the $a_i$ generate $\Z^2$ as a group, then $S$ consists of all lattice points in a certain polygon, with the exception of elements with bounded distance from the boundary of the polygon. 
Hence, it suffices to find a lattice point with both coordinates divisible by $p$, which is in the interior of the convex hull $\overline{S}$ of $S$, and has a sufficient distance from the boundary of 
$\overline{S}$. 

The convex hull $\overline{S}$ can be described as $\{\sum_{i=1}^k t_i a_i:t_i\in[0, m_i,] \sum_{i=1}^k t_i = p\}$. We can further simplify the problem by rescaling. To get rid of all references to $p$, we allow arbitrary non-negative real multiplicities $\mu_i$. Formally we define a real sequence $A$ over a group $G$ as a finite or infinite sequence $((a_i, \mu_i))$ of pairs of distinct elements $a_i$ of $G$ and non-negative real numbers $\mu_i$. We call $\mu_i$ the muliplicities of $a_i$, and $|A|=\sum \mu_i$ the length of $A$. If $X\subseteq G$, we say that $X$ contains $\sum_{i:a_i\in X} \mu_i$ elements of $A$. We put $\Sigma_\R^1(A) = \{\sum t_i a_i:t_i\in[0, \mu_i,] \sum t_i = 1\}$.


\begin{Prop}
\label{Prop:Dreal}
There exists a constant $c>0$, such that the following holds.
If $A\subseteq\Z^2$ is a real sequence with $|A|=3.99$, maximal
multiplicity $\leq 1/2$, such that no line contains more than 3/2
points. Then after a suitable linear transformation preserving the lattice $\Z^2$, there exists a lattice point $x$, such that $B_{c}(x)\subseteq \Sigma_\R^1(A)$.
\end{Prop}
As the set of all real sequences of length $\leq 3.99$ with support in $([-C, C]\cap\Z)^2$, no multiplicity $>\frac{1}{2}$ and no line containing more than $\frac{3}{2}$ points is a compact subset of $\R^{(2C+1)^2}$, and the function mapping a sequence $A$ to the maximal $r$ such that $B_r(x)\subseteq\Sigma_\R^1(A)$ for some lattice point $x$ is continuous, it suffices to show that for every admissible real sequence $A$ we have that $\Sigma_\R^1(A)$ contains an interior lattice point.

For the proof it suffices to consider configurations such that $\Sigma_\R^1(A)$ is minimal with respect to $\subseteq$ among all real sequences $A$ satisfying these conditions. In the sequel we will only consider such minimal configurations.
\begin{Lem}
\label{Lem:DReplace}
Suppose there is a counterexample to Proposition~\ref{Prop:Dreal}. Then there is a counterexample for which the following holds true.
Suppose that $v_1, v_2, v_3, v$ are distinct lattice points with $\mu(v_1), \mu(v_2), \mu(v_3)$ posisitive, and suppose that $v$ lies in the
interior or on the boundary of the triangle with vertices $v_1, v_2,
v_3$. Then $\mu(v)=1/2$, or $v$ lies on a line $\ell$ which contains $3/2$ points. Furthermore, if $v$ is on the boundary of the triangle $v_1v_2v_3$, then $\ell$ is not the line supporting the side of the triangle containing $v$.
\end{Lem}
\begin{proof}
Suppose this is not true. Let $\mathcal{L}$ be the set of all lines passing through $v$, and put $m=\min(\mu(v_1), \mu(v_2), \mu(v_3), \frac{1}{2}-\mu(v), \frac{3}{2}-\mu(g):g\in\mathcal{L})$. By assumption we have that $m$ is positive. There exist non-negative real numbers $s_1. s_2, s_3$, such that $s_1+s_2+s_3=1$ and $s_1v_1 + s_2v_2+s_3v_3=v$. Form the real sequence $A'$ by reducing the multiplicity of $v_i$, $\nu=1,2,3$ by $s_im$, and increasing the multiplicity of $v$ by $m$. By our choice of $m$ we have that $A'$ satisfies the assumptions of the lemma. Furthermore any subsequence sum of length 1 in $A'$ can be represented as a subsequence sum of length 1 in $A$, hence $\Sigma_\R^1(A')\subseteq\Sigma_R^1$, hence, if $A$ is a counterexample so is $A'$. Repeating this process for all interior points in the convex hull of $A$ our claim follows.

If $v_1, v_2, v$ are on one line $\ell$, such that $\mu(\ell)=\frac{3}{2}$, and neither $\mu(v)=\frac{1}{2}$ nor there exists a second line $\ell'$ passing through $v$ with $\mu(\ell')=\frac{3}{2}$, then building the sequence $A'$ as above does not change $\mu(\ell)$, as $\mu(v)$ is increased y exactly the same amount $\mu(v_1)+\mu(v_2)$ is decreased, hence, t$A'$  still satisfies the conditions of the proposition, and we can argue as before.
\end{proof}

\begin{Lem}
\label{Lem:DPrimitive}
Suppose that $v_1, v_2$ are lattice points with $t_1, t_2=1/2$, and assume that $v_1-v_2$ is not 
primitive. Then $S$ contains an interior lattice point.
\end{Lem}
\begin{proof}
Without loss we may assume that $v_1=(0, 0)$, $v_2=(k, 0)$ with $k\geq 2$. By Lemma~\ref
{Lem:DReplace} we have that $(1, 0)$ has multiplicity $1/2$ as well, or is on a line different from 
$y=0$ containing 3/2 points. In the latter case we immediately see that $S$ contains an interior 
lattice point on this line, hence, from now on we assume that $(1, 0)$ has multiplicity 1/2. If there 
are points $a, b$ with positive multiplicity such that $a$ has positive, and $b$ has negative 
second coordinate, then $(1, 0)$ is in the interior of $S$, and we are done. There are at least 1/2 
points with second coordinate $\geq 1/2$, hence, we can represent some point $(u, 1)$ as a 
subset sum of such points of length $\leq 1/2$. Adding 1/2 copies of points on the line $y=0$ we 
find that the segment $(u, 1)$ to $(u+1, 1)$ is contained in $S$, and this segment contains an 
interior lattice point, unless $u$ is integral. If $u$ is a lattice point, we are stizlle done, unless $(u, 
1)$ and $(u+1, 1)$ are both on the boundary of $S$. But this is only possible if all points with 
positive multiplicity and second coordinate $\geq 2$ are equal to one such point $(t, 2)$ with 
multiplicity 1/2. But then $A$ contains 13/4 points on two lines, contradicting the assumption that 
no line contains more than 3/2 points.
\end{proof}
\begin{Lem}
\label{Lem:Dnoline}
Suppose there exists a line $\ell$ containing 3/2 points. Then $S$ contains an interior lattice point.
\end{Lem}
\begin{proof}
If  there exists a point $x\in \ell$ and some $\epsilon>0$, such that $B_\epsilon(x)\cap\ell\subseteq\Sigma_\R^1(\ell)$. Then there exists some $\delta>0$, such that $B_\delta(x)\subseteq \Sigma_\R^1(A)$, unless all elements of $A$ are on one side of $\ell$. We now assume that this is true, without loss of generality we assume that $\ell=\{(x,0)\}$ and that $A$ is supported on $\{x,y):y\geq 0\}$. We can partition $A\cap\ell$ into three subsets $A_1, A_2, A_3$, such $\mu(A_i)=\frac{1}{2}$ that for $(x_i, 0)\in A_i$ we always have $x_1\leq x_2\leq x_3$. Let $s_i$ be the sum of all elements in $A_i$. Then the convex hull of $\{s_1, s_3\}$ is an interval of length at least 1. As there are at most 3 elements with second coordinate 0 or 1, there are at least $0.99$ points with second coordinate $\geq 2$. Put $S=\Sigma_\R^{1/2}(A\cap\{(x,y):y\geq 2\})$. Then $S$ is a convex subset of $\{(x,y):y\geq 1\}$. By forming a convex combination with elements in $A_2$, we see that $\Sigma_\R^{1/2}((A\cap\{(x,y):y\geq 2\})\cup A_2)$ contains the convex hull of $S$ and $s_2$. Adding a convex combination of $s_1$ and $s_3$ to an element of $\Sigma_\R^{1/2}((A\cap\{(x,y):y\geq 2\})\cup A_2)$ we obtain that $\Sigma_\R^1(A)$ contains an interval of length at least 1 on the line $\{(x,1)\}$. As $\Sigma_\R^1(A)$ contains points above and below this line, we obtain an interior lattice point, unless this interval has length exactly 1, and its endpoints are lattice points. This is only possible if at most one point on the line $\{(x, 1)\}$ has positive multiplicity, and $A\cap\{(x,y):y\geq 2\}$ is supported on a single line. passing through this point. However, then $A$ is supported on the union of two lines, contradicting $|A|>3$. 

Now assume that there is no lattice point $x$ as above, and $\ell$ intersects the interior of the convex hull of the support of $A$. As $|A|=\frac{3}{2}$ and the maximal multiplicity is $\leq\frac{1}{2}$, this is only possible if there are four points with positive multiplicity on this line, and the multiplicities of these points are $\frac{1}{4}, \frac{1}{2}, \frac{1}{2}, \frac{1}{4}$, we may assume that these points are $(0,0) - (3,0)$. 

Let  $(x, 0)$ be an element in the convex hull of the support of $A\setminus\ell$. Such an $x$ exists, and by symmetry we may assume that $x\leq\frac{3}{2}$. Then there is some $c\in(0,\frac{1}{4})$, such that $(x,0)\in\Sigma_\R^c(A\setminus\ell)$, thus, 
\[
\frac{1}{4}\cdot (0,0) + \frac{1}{2}(1,0) + c(x,0) + \left(\frac{1}{4}-c\right)(2,0) = (1-(2-x)c, 0)
\]
is in $\Sigma_\R^1(A)$, and $(1,0)$ is an interior lattice point.
\end{proof}

\begin{Lem}
There exist three affinely independent elements of $A$ which have multiplicity $\frac{1}{2}$.
\end{Lem}
\begin{proof}
If there are three elements in $A$ which have multiplicity $\frac{1}{2}$, but which are not independent, then they lie on a line, which has multiplicity $\frac{3}{2}$, but we have already seen that this is impossible. Hence, it suffices to show that there are at least three elements with multiplicity $\frac{1}{2}$.

If there are 9 different points with positive multiplicity, then there are three points $v_1, v_2, v_3$, which are congruent to each other modulo 2. If these points are not on one line, then the midpoints $m_1, m_2, m_3$ of the line segments between these points are not on a line either, and by Lemma~\ref{Lem:DReplace} each of $m_1, m_2, m_3$ has multiplicity $\frac{1}{2}$, and we are done. If $v_1, v_2, v_3$ are all on a line, and $v_2$ is between $v_1, v_2$, then the midpoints of $v_1v_2$, $v2$, and the midpoint of $v_2v_3$ are three different points each of which has multiplicity $\frac{1}{2}$, which is impossible. Hence, in this cae our claim holds.

If there are 8 different poitns with positive multiplicity, then, as each point has multiplicity $\leq\frac{1}{2}$, and the total multiplicity is $3.99$, we have that each point has multiplicity $\geq 0.49$. If among the 8 points there are three which are equal modulo 2, we argue as in the case of nine points, hence, we may assume that the 8 points form 4 pairs which are congruent modulo 2. The midpoints of the four segments connecting the points in a pair cannot be all distinct, for otherwise we would have 4 points with multiplicity $\frac{1}{2}$. Hence, there is some point which is the midpoint of two of these segments. After a suitable linear transformation we find that we can assume that each of $(-1,0), (0,0), (1,0), 0,-1), (0,1)$ occurs with multiplicity $\geq 0.49$. But then $(0,0)$ is an interior point of $\Sigma_\R^1(A)$, and we are done.
\end{proof}
We may assume that $(0,0), (1,0),$ and $(0,1)$ occur with multiplicity $\frac{1}{2}$. If tow further points with multiplicity $\frac{1}{2}$, there would be two points with multiplicity $\frac{1}{2}$ which are congruent modulo 2, which contradicts Lemma~\ref{Lem:DPrimitive}. If there is one further point with multiplicity $\frac{1}{2}$, it must be congruent to $(1,1)$ modulo 2, and the convex hull of $(0,0), (1,0), (0,1)$ and this fourth point does not contain any other lattice point. This is only possible if the fourth point is $(\pm 1,\pm 1)$. If this point is $(-1,-1)$, then $(0,0)$ is an interior point of $\Sigma_\R^1(A)$, and the other three cases are equal up to a linear transformation. We can therefore assume that either there is no further point with multiplicity $\frac{1}{2}$, or this point is $(1,1)$. We now distinguish three cases depending on the multiplicity of $(1,1)$.

\begin{Lem}
\label{Lem:no four}
Suppose that $(0,0), (1,0), (0,1)$ and $(1,1)$ occur with multiplicity $\frac{1}{2}$. Then one of $(0,0), (1,0), (0,1), (1,1)$ is an interior point of $\Sigma_\R^1(A)$.
\end{Lem}
\begin{proof}
If $(u,v)$ occurs with positive multiplicity, then we can apply Lemma~\ref{Lem:DReplace} to $(u,v)$ and $(u\bmod 2, v\bmod 2)$ to find that $m=\left(\frac{u+(u\bmod 2)}{2}, \frac{v+(v\bmod v)}{2}\right)$ has multiplicity $\frac{1}{2}$. As there are only 4 points with multiplicity $\frac{1}{2}$, $m$ has to be one of them, and we see that $A$ is supported on$\{(x,y):-1\leq x,y\leq 3\}$.

Suppose that $\mu((-1,0))+2\mu((-1,-1))+\mu((0, -1))>\frac{1}{4}$. Then $\Sigma_\R^1(A)$ contains a point in $\{(x,y):x,y\leq 0\}$ different from $(0,0)$. If this point has both coordinates different from 0, then $(0,0)$ is an interior point of $\Sigma_\R^1$. If we cannot arrange for this point to have both coordinates different from 0, then without loss of generality we may assume that $A$ contains no element on the line $\{(x,y):y=-1\}$, and $(-1,0)$ with multiplicity $>\frac{1}{4}$. Then there are $\geq0.99$ points on the line $\{(x,y):y=2\}$. If $2\mu((-1,2))+\mu((0,2))\geq\frac{1}{4}$, then $(0,1)$ is an interior point, hence, $\mu((1,2))+\mu((2,2))\geq 0.74$. But then $(1,1)$ is an interior point.

Hence, we obtain in particular $\mu((-1,0))+\mu((-1,-1))+\mu((0, -1))\leq\frac{1}{4}$. Applying the same argument to the other points of multiplicity $\frac{1}{2}$ we obtain $\mu(A\setminus\{(0,0),(1,0), (0,1),(1,1)\})\leq 1$, thus, $|A|\leq 3$, which is impossible.
\end{proof}

\begin{Lem}
\label{Lem:no 3.5}
Suppose that $(0,0), (1,0)$, and $(0,1)$ occur with multiplicity $\frac{1}{2}$, and $(1,1)$ occurs with multiplicity strictly between $0$ and $\frac{1}{2}$ . Then one of $(0,0), (1,0), (0,1)$ is an interior point of $\Sigma_\R^1(A)$.
\end{Lem}
\begin{proof}
As in Lemma~\ref{Lem:no four} we see that $A$ is supported on$\{(x,y):-1\leq x,y\leq 3\}$, and that the total multiplicity of points in the region $\{(x,y):x,y\leq 0\}$ different from $(0,0)$ is $\leq\frac{1}{4}$. 

If $\mu((1,-1))+\mu((2,0))>\frac{1}{4}$, then $\Sigma_\R^1(A)$ contains $(t,0)$ for some $t>1$. Then $(1,0)$ is an interior point of $\Sigma_\R^1(A)$, unless $\mu((1,-1))=\mu((2,-1))=0$. So $\mu((1,-1))+\mu((2,-1))+\mu((2,0))$ either equals $\mu((2,0))<\frac{1}{2}$, or is bounded by $\mu((1,-1))+\frac{1}{4}\leq\frac{3}{4}$. Hence, we obtain
\begin{multline*}
|A| = \underbrace{\mu((0,0))+\mu((1,0))+\mu((0,1))+\mu((1,1))}_{<2} + \underbrace{\mu((1,-1))+\mu((2,-1))+\mu((2,0))}_{\leq\frac{3}{4}}\\
 + \underbrace{\mu((-1,1))+\mu((-1,2))+\mu((0,2))}_{\leq\frac{3}{4}} + \underbrace{\mu((-1,-1))+\mu((-1,0))+\mu((0,-1))}_{\leq\frac{1}{4}}\\
  <3.75,
\end{multline*}
a contradiction.
\end{proof}

We can now finish the proof of Proposition~\ref{Prop:Dreal}. If $\mu((-1,1))>0$, we can transform $A$ to the situation dealt with in Lemma~\ref{Lem:no 3.5}, and similarly if $\mu((1,-1))>0$. As before we have that $A$ is supported on $\{(x,y):-1\leq x,y\leq 2\}$. As $\mu((1,1))=0$, we have also  $\mu((1,2))=\mu((2,1))=\mu(2,2))=0$, and at most one of $\mu((2,0))$ and $\mu((0,2))$ is different from 0.Finally, we have $\mu((-1,0))+\mu((-1,-1))+\mu((0,-1))\leq\frac{1}{4}$, or $(0,0)$ is an interior point of $\Sigma_\R^1(A)$. Hence,
\begin{multline*}
|A|\leq \underbrace{\mu((0,0))+\mu((1,0))+\mu((0,1))}_{=\frac{3}{2}} + \underbrace{\mu((2,0))+\mu((0,2))}_{\leq\frac{1}{2}}\\
 + \underbrace{\mu((-1,0))+\mu((-1,-1))+\mu((0,-1))}_{\leq\frac{1}{4}} \leq \frac{9}{4},
\end{multline*}
which gives a contradiction.

Hence, the proof of Proposition~\ref{Prop:Dreal} is complete.

\section{Conclusion of the proof}

Let $A$ be a sequence of length $4p-4$, and suppose that the maximal multiplicity of an element in $A$ is $\leq\frac{p}{2}$. Then either there exists a linear transform such that all but $0.003p$ elements of $A$ have coefficients in $[-C, C]^2$ for some absolute constant $C$, or our claim follows from Section~\ref{sec:Fourier}. Removing all elements outside $[-C, C]^2$ we obtain a sequence $A_1$ of length $\geq 3.997p$ supported in $[-C, C]$. Removing some further elements we obtain a sequence $A_2$ such that all elements have either multiplicity 0 or multiplicity $\geq \frac{0.004 p}{(2C+1)^2}$. We now interpret $A_2$ as a sequence over $\Z^2$. If $A_2$ does not generate $\Z^2$ as an affine lattice, we can apply some linear transformation such that the image generates $\Z^2$, while the elements are still bounded, so without loss of generality we may assume that $A_2$ does generate $\Z^2$. Pick a generating set $D$ of $\Z^2$ in the support of $A$ and let $A_3$ be the subsequence obtained by removing $\lfloor \frac{0.003p}{(2C+1)^2}\rfloor$ copies of each element in $D$. 

Now form the real sequence $\overline{A_3}$ which has the same support as $A_3$, but all multiplicities are divided by $p$. Then $|\overline{A_3}|\geq 3.99$, and no element has multiplicity $>\frac{1}{2}$. It follows from Section~\ref{sec:geo} that there exists some lattice point $x$ and a constant $c>0$ such that $B_c(x)\subseteq\Sigma_\R^1(\overline{A_3})$. Clearly, the coordinates of all elements in $\Sigma_\R^1(\overline{A_3})$ are at most $C$. Therefore,  multiplying all multiplicities in a representation of $y\in B_c(x)$ by $1-\frac{2c}{3C}$ we obtain a representation of $(1-\frac{2c}{3C})y$ as the sum of a subsequence of length $1-\frac{2c}{3C}$. We conclude that $B_{2c/3}(x)\subseteq \Sigma_R^{1-\frac{2c}{3C}}(\overline{A_3})$.

Translating this back to the original setting we obtain that for every lattice point $y\in B_{2cp/3}(px)$ there exist real numbers $t_i$, $1\leq i\leq k$ and elements $a_i$ in the support of $A_3$, such that $\sum_{i=1}^k t_i a_i=y$, and $\sum_{i=1}^k t_i=\left(1-\frac{2c}{3C}\right)p$. We can pick integers  $\widetilde{t_i}$, such that $|t_i-\widetilde{t_i}|<1$ and $\sum_{i=1}^k\widetilde{t_i} = \lceil\left(1-\frac{2c}{3C}\right)p\rceil=:N$. Then $\widetilde{y} = \sum_{i=1}^k \widetilde{t_i} a_i$ satisfies 
\[
|y-\widetilde{y}| = \left|\sum_{i=1}^k (t_i-\widetilde{t_i})a_i\right| \leq \sum_{i=1}^k |t_i-\widetilde{t_i}|\cdot |a_i| \leq \sum_{i=1}^k C \leq C(2C+1)^2\leq 5C^3.
\]
We conclude that for every lattice point $y\in  B_{2cp/3}(px)$ there exists a lattice point $\widetilde{y}$ with $|y-\widetilde{y}|\leq 5C^3$, such that $\widetilde{y}$ can be represented as the sum of a subsequence of $A_3$ of length $N$. 

As $D$ is a generating set of $\Z^2$, and all elements of $D$ have all coordinates bounded by $C$, we have that $A_4 = D^{\lfloor\frac{0.003}{(2C+1)^2}\rfloor}$ is a sequence such that the set of all elements that are sums of subsequences of length $p-N = \frac{pc}{2C}+\mathcal{O}(1)$ equals the set of all lattice points in $\Sigma_\R^N(A_4)$, apart from a certain set of lattice points with distance at most $C^2$ from the boundary of $\Sigma_\R^N(A_4)$. We conclude that if $p$ is sufficiently large, there exist integers $u, v$, such that $|u|, |v|\leq (p-N)C =\frac{cp}{3C}+\mathcal{O}(1)$ and $(u+i, v+j)\in\Sigma^N(A_4)$. It now follows that $\Sigma^p(A)\supseteq\Sigma^{p-N}(A_3)+\Sigma^N(A_4)$ contains a lattice point with both coordinates divisible by $p$, and the proof is complete.

\end{document}